\documentclass[11pt,a4paper]{amsart}
\usepackage{amsthm,amsmath}
\usepackage[all]{xy}
\title{A structure theorem and the graded Betti numbers for almost complete intersections} 
\author{Alfio Ragusa
	\and  Giuseppe Zappal\`a}

\subjclass[2000]{13 D 40, 13 H 10}

\DeclareSymbolFont{rsfscript}{OMS}{rsfs}{m}{n}
\DeclareSymbolFontAlphabet{\mathrsfs}{rsfscript}

\DeclareSymbolFont{AMSb}{U}{msb}{m}{n}
\DeclareSymbolFontAlphabet{\mathbb}{AMSb}

\DeclareSymbolFont{eufrak}{U}{euf}{m}{n}
\DeclareSymbolFontAlphabet{\gothic}{eufrak}

\def\ac{\`}

\newcommand{\f}{\footnotesize}
\newcommand{\aci}{{\it aci}}

\newcommand{\pp}{\mathbb P}

\newcommand\depth{\operatorname{depth}}
\newcommand\hd{\operatorname{hd}}

\newcommand\CI{\operatorname{CI}}
\newcommand\Supp{\operatorname{Supp}}
\newcommand\Nm{{\mathbb N}_{\le}}

\newcommand\pf{\operatorname{pf}}
\newcommand\Pf{\operatorname{Pf}}
\newcommand\idt{\operatorname{id}}

\newcommand\rw{\Rightarrow}

\newcommand\mci{\operatorname{mci}}
\newcommand\Mng{\operatorname{Mng}}

\newcommand\im{\operatorname{im}}

\newtheorem{thm}{Theorem}[section]
\newtheorem{lem}[thm]{Lemma}
\newtheorem{prp}[thm]{Proposition}
\newtheorem{cor}[thm]{Corollary}

\theoremstyle{definition}
\newtheorem{dfn}[thm]{Definition}

\theoremstyle{remark}
\newtheorem{rem}[thm]{Remark}
\newtheorem{exm}[thm]{Example}

\begin{document}


\begin{abstract}
We provide a structure theorem for all  almost complete intersection ideals of depth three in any Noetherian local ring. In particular, we find that the minimal generators are the pfaffians of suitable submatrices of an alternating matrix. From the graded version of the previous result, we characterize the graded Betti numbers of all $3$-codimensional almost complete intersection schemes of $\pp^r.$
\end{abstract}

\maketitle

\section*{Introduction}
\markboth{\it Introduction}{\it Introduction}
It is well known that in any regular local ring $R$ every ideal $I$ for which the number of elements in a minimal set of generators is equal to its heigth is generated by a regular sequence, in particular it is a perfect ideal (in the sense that $\hd R/I= \depth I$). This simple fact has an important consequence for studying projective schemes $X$ of $ \pp^r_k,$ since this class of ideals characterizes complete intersection schemes, i.e. schemes $X$ of codimension $c$ whose defining (saturated) ideal $I_X$ in $R=k[x_0,x_1,\ldots,x_r]$ is generated by $c$ elements, hence schemes which can be described with the minimum number of equations.  Thus, complete intersection schemes are necessarily generated by a regular sequence and consequently are arithmetically Cohen Macaulay (aCM). In particular, from this one has a very simple description of a minimal free resolutions of $R/I_X$ (Koszul resolution) and consequently of the graded Betti numbers. The question becomes immediately more complicated when one admits that the ideal $I_X$ has one generator more than its depth. In some part of the literature aCM schemes $X$ of $ \pp^r_k,$ of codimension $c$ whose defining ideal $I_X$ is  minimally generated by $c+1$ elements are called {\em almost complete intersection} schemes. Similarly, in any  unitary commutative local ring $A$ a perfect ideal $I$ for which $\nu(I)=\depth I+1$ is said to be an {\em almost complete intersection} ideal.  Some discussion about almost complete intersection  schemes can be found on the J. Migliore book's \cite{Mi}. Very little is known about almost complete intersection schemes (or ideals), for instance some result for \lq\lq generic\rq\rq\ almost complete intersections can be found on the article of  J. Migliore and R. Mirò Roig \cite{MM}. The very first observation which makes these ideals (or these schemes) nice to study is the fact that every almost complete intersection is directly linked in a complete intersection to a Gorenstein ideal (or aG scheme). Indeed,  if $I_Q \subset R$ is the defining ideal of an almost complete intersection of codimension $c$ and $I_Z \subseteq I_Q$ is generated by $c$ minimal generators of $I_Q$ which form a regular sequence  then $I_G:=I_Z:I_Q$  is the defining ideal of an aG  scheme. By liaison theory (see for instance \cite{PS}) we have also $I_Q=I_Z:I_G.$  Now, using this fact and the very nice structure theorem for aG schemes of codimension $3$ (Buchsbaum-Eisenbud \cite{BE}), in this paper we will give a structure theorem for all almost complete intersection of depth $3$ (see Theorem \ref{tqci}). This result generalizes an analogous one of S. ~Seo which, in \cite{Se} Theorem 2.4, gives a similar characterization for those almost complete intersections $Q$  such that  the aG scheme defined by $I_G:=I_Z:I_Q$ has the three minimal generators of smallest degree performing a regular sequence. On the other hand, by Diesel paper (see \cite{Di}), one can characterize all graded Betti numbers for $3$-codimensional aG schemes. Then, using this characterization and the mentioned structure theorem, we obtain the main result of this paper, i.e a charaterization of  all graded Betti numbers for $3$-codimensional almost complete intersections (Theorem \ref{betti}). Thus, this result provides a new large class of $3$-codimensional projective schemes, besides Gorenstein schemes, for which one is able to give a complete description of the graded Betti numbers.
 
\section{Notation and preliminaries} 
\markboth{\it Notation and preliminaries}
{\it Notation and preliminaries}

Let $R$ be a Noetherian local ring. A perfect ideal $I$ of $R$ is called an {\em almost complete intersection} ideal if $\nu(I)=\depth I+1$ (here $\nu(I)$ denotes the number of minimal generators of $I$). Of course, in case $R$ is Cohen-Macaulay an ideal $I$ is an almost complete intersection when $\nu(I)=ht(I)+1,$ where  $ht(I)$ is its height or codimension. Analogously, let $k$ be an algebraically closed field and $X \subset \pp^r_k$ be a closed subscheme of codimension $c.$ Let $I_X\subset R=k[x_0,x_1,\ldots,x_r]$ be the saturated homogeneous ideal defining $X.$  Then $X \subset \pp^r_k$ is said an {\em almost complete intersection} scheme when its defining ideal $I_X$ is perfect and minimally generated by $c+1$ elements. Our first observation is that every almost complete intersecion is directly linked in a complete intersection to an arithmetically Gorenstein (aG) scheme. Indeed,  if $I_Q \subset R$ is the defining ideal of an almost complete intersection of heigth $c$ and $I_Z \subseteq I_Q$ is generated by $c$ minimal generators of $I_Q$ which form a regular sequence  then $I_G:=I_Z:I_Q$  is the defining ideal of an aG  scheme. By liaison theory (see \cite{PS}) for a complete discussion on this argument) we have also $I_Q=I_Z:I_G.$ 

Therefore many properties of almost complete intersections can be deduced by properties of aG schemes. Since in case of codimension $c=3$ we have the Eisenbud Buchsbaum structure theorem for Gorenstein algebras (see \cite{BE}), in this paper we will produce an analogous structure theorem for all almost complete intersections of depth $3$. We will give also a graded version of this structure theorem. Moreover, since again in codimension $c=3$  by Diesel paper (see \cite{Di}) all graded Betti numbers for Gorenstein graded algebras are characterized, we will describe all possible Betti sequences for almost complete intersections of depth $3$. For all of this we need to fix some terminology and to remind some known facts about free resolutions, in particular, for Gorenstein algebras and about pfaffians of alternating matrices. Moreover, since the Betti sequences are simply sequences of finite multisets of positive integers, we need also to fix terminology and basic facts on multisets. 

We start just with basic notation and properties of multisets. 

Let $A$ be a set. A {\em multiset} on $A$ is a function $M:A \to \mathbb{Z}^+;$ for every $a \in A$ the integer $M(a)$ will be called the {\em multiplicity} of $a$ in $M$ and we will denote it also with $\mu_M(a).$ When $M$ is a multiset the domain of $M$ is called the {\em support} of $M$ and we denote it by $\Supp M.$ Whenever $b \notin \Supp M$ we will set $\mu_M(b)=0.$

Sometimes, for simplicity, we will denote the multiset $M$ with a finite support by $M= \{\{m_1,m_2, \ldots, m_t\}\},$ where in the list each element of the support of $M$ appears as many times as its multiplicity in $M,$ and sometimes we will use the notation  $M= \left\{a_1^{(\mu_M(a_1))},a_2^{(\mu_M(a_2))}, \ldots, a_r^{(\mu_M(a_r))}\right\},$ where $a_i$ are all the elements in $\Supp M.$ We define $|M|=\sum_{a \in \Supp M}\mu_M(a).$ Of course, if $M$ is a multiset with $\mu_M(a)=1$ for all elements in $\Supp M$ we can identify the multiset $M$ with its support.

Now we will remind the definitions concerning the main operations on multisets. If $M$ and $N$ are two multisets we will denote
$$M \cap N: \Supp M \cap \Supp N \to \mathbb{Z}^+,  \quad \textrm{where}  \quad (M \cap N)(y)=\min\{\mu_M(y),\mu_N(y)\},$$
$$M \cup N: \Supp M \cup \Supp N \to \mathbb{Z}^+,  \quad \textrm{where}  \quad (M \cup N)(y)=\max\{\mu_M(y),\mu_N(y)\},$$
$$M \sqcup N: \Supp M \cup \Supp N \to \mathbb{Z}^+,  \quad \textrm{where}  \quad (M \sqcup N)(y)=\mu_M(y)+\mu_N(y),$$
if we set $\Supp (M\setminus N)=\{x \in \Supp M \ |\ \mu_M(x)> \mu_N(x)\}$  then
$$M \setminus N: \Supp (M\setminus N) \to \mathbb{Z}^+,  \quad \textrm{where}  \quad (M \setminus N)(y)=\mu_M(y)-\mu_N(y);$$
moreover, we will say that $M$ is a submultiset of $N$ and we will write $M \subseteq N$ whenever $\Supp M \subseteq \Supp N$ and for every $y \in \Supp M$ $\mu_M(y)\le \mu_N(y).$  Of course, when $M \subseteq N$ then $N=M \sqcup (N\setminus M).$

In the sequel when $M$ is a finite multiset of integers we define $\|M\|=\sum_{y\in \Supp M}\mu_M(y)y.$ Furthermore, if $n$ is also an integer, we define the multiset $n\pm M$ by 
$$n\pm M: n\pm \Supp M \to \mathbb{Z}^+,  \quad \textrm{where} \quad  (n\pm M)(n\pm y)=\mu_M(y) .$$

\begin{prp}\label{dual}
If $M$ is a multiset of integers, $n$ is  an integer and $H:=M \cap (n-M),$ then  for every $x\in \Supp H$ $\mu_H(x)= \mu_H(n-x).$
\end{prp}
\begin{proof} Of course, for every $x \in \Supp M$ $\mu_M(x)= \mu_{n-M}(n-x).$ Then for every  $x\in \Supp H$ we have 
\begin{multline*}\mu_H(x)=\min\{\mu_M(x),\mu_{n-M}(x)\}=\\ =\min\{\mu_{n-M}(n-x),\mu_{M}(n-x)\}=\mu_H(n-x).\end{multline*}
\end{proof}

Let now $R=k[x_0,\ldots,x_r]$ and $I_X$ the defining (saturated) ideal of a projective scheme $X \subset \pp^r_k.$ We recall that the  standard graded $k$-algebra $R/I_X$ admits a graded minimal free resolution of the following type
$$0\to\bigoplus_jR(-j)^{\beta_{pj}}\to\ldots\to\bigoplus_jR(-j)^{\beta_{0j}}\to R\to R/I_X\to 0$$
which, if we restrict ourselves to the $\beta_{ij}\ne 0$, can be written as
 $$0\to\bigoplus_{h\in \Supp M_p}R(-h)^{\mu_p(h)}\to\ldots\to\bigoplus_{h\in \Supp M_1}R(-h)^{\mu_1(h)}\to R\to R/I_X\to 0$$
where each $M_i$ is a multiset of positive integers and $\mu_i(h):=\mu_{M_i}(h).$  
 We will call $(M_1,\ldots,M_p)$ the {\em Betti sequence} of $X$ or also of $R/I_X,$ and we will denote it by $\beta_X$ or $\beta_{R/I_X}.$ 

For simplicity, from now on for every multiset of positive integers $M$ we will set 
$$\bigoplus_{h\in M}R(-h):=\bigoplus_{h\in \Supp M}R(-h)^{\mu_M(h)}$$
hence, the resolution can be written 
$$0\to\bigoplus_{h\in M_p}R(-h)\to\ldots\to\bigoplus_{h\in M_1}R(-h)\to R\to R/I_X\to 0$$
Now, since Gorenstein algebras of depth $3$ can be described by alternating matrices, we establish some terminology  for such matrices. 

Let $R$ be any commutative ring. 

If $i\ne j$ are positive integers we set $\langle i,j \rangle$ the following integer
$$\langle i,j \rangle=\begin{cases} i+j+1 & \text{if} \ i<j \\
i+j & \text{if} \ i>j\end{cases} $$
If $M=(a_{ij})$ is an alternating matrix of size $m$ with entries in $R$ we will denote by $M_{\widehat{ij}}$ the alternating matrix obtained from $M$ by deleting both rows and columns $i$ and $j.$ With this terminology when $m$ is even one can easily verify that the pfaffian of $M$ (for definitions and basic facts on phaffians see for instance \cite{IK} Appendix B) can be computed, for every $i=1, \ldots, m,$ by
$$\pf M= \sum_j (-1)^{\langle i,j \rangle}a_{ij}\ \pf M_{\widehat{ij}}.  \quad \quad (1)$$

Moreover, if  $M$ is an alternating matrix of size $m,$ with $m$ even, we will denote by $\overline{M}=(\overline{a}_{ij}),$ where  $\overline{a}_{ij}=(-1)^{\langle i,j \rangle}\pf M_{\widehat{ij}},$ the pfaffian adjoint of $M,$ which is clearly an alternating matrix.

\begin{rem}\label{lpl}
If 	
$$M=\begin{pmatrix}
0 & a & | & B \\
-a & 0 & | & \\
-- & -- & -|-  & --\\
-{}^tB & & | & C\\	
\end{pmatrix}$$
is an even alternating  matrix, by repeating the previous formula for $i=1,2,$ one can show that its pfaffian can be computed as
$$\pf M =a\ \pf C + \pf (B \ \overline{C}\ {}^tB)$$
where $\overline{C}$ is the pfaffian adjoint of $C$ as defined above.
\end{rem}

Let $\psi: F^{\vee} \to F$ be an alternating map with $F$ a free $R$-module of rank even $m.$  If $B$ is a basis for $F$ and $B^{\vee}$ the dual basis for $F^{\vee},$ if $M$ is the matrix associated to $\psi$ with respect to such bases, we will denote by $\overline{\psi}: F \to F^{\vee}$  the map whose associated matrix with respect to the previous bases is the pfaffian adjoint of $M,$ i.e. $M(\overline{\psi})=\overline{M}.$  Note that $\overline{\psi}\psi= (\pf \psi) \idt_{F^{\vee}}$ and $\psi \overline{\psi}= (\pf \psi) \idt_{F}.$ Moreover we will write $\Pf_{s}(\psi)$ for the ideal generated by th pfaffiani di $\psi$ of order $s.$

Furthermore, with the same notation, when  the rank $m$ is odd and $\{p_1, \ldots, p_m\}$ are the $m$ submaximal pfaffians of the matrix $M=M(\psi)$ (with respect to fixed bases $B$ and $B^{\vee}$), if $p \in I=(p_1, \ldots, p_m),$ then there exists an alternating map $\widetilde{\psi}: G^{\vee} \to G,$ where $G$ is a free module of rank $m+2,$  such that the alternating matrix $\widetilde{M}=M(\widetilde{\psi})$ (with respect to suitable bases $\widetilde{B}$ and $\widetilde{B}^{\vee}$) has the $m+2$ submaximal pfaffians $\{p_1, \ldots, p_m, p, 0\}.$ Indeed, if $p=a_1p_1+ \ldots + a_mp_m,$ it is enough to take $G=F\oplus R^2$ and as $\widetilde{\psi}: G^{\vee} \to G,$ the alternating map which, respect to  given bases $\widetilde{B}$ and $\widetilde{B}^{\vee},$ is defined by the alternating matrix 
$$\widetilde{M}=\begin{pmatrix}
 &  &  & | & 0 & a_1 \\
 & M & & | & \vdots & \vdots\\
  &  & & | & 0 & a_m \\
-- & -- &-- & --|-- & --& -- \\
0 & \cdots & 0 & | & 0 & -1 \\
-a_1 & \cdots  & -a_m & | & 1 & 0 \\	
\end{pmatrix}$$
To complete the assertion it is enough to compute the submaximal pfaffians of $\widetilde{M}$, using the above formula $(1)$ with respect to the $(m+1)$-th column for all pfaffians except for the $(m+1)$-th pfaffian  for which we use formula $(1)$ with respect to the last column. Applying the previous observation we get the following 

\begin{lem}\label{3gen}
Let $R$ be a noetherian local ring and $I_G$ a Gorenstein ideal of depth $3$ in $R$ and $p_1,p_2,p_3\in I_G.$ Then there exists an alternating map 
$\varphi: H^{\vee} \to H$ of odd rank $m$ such that $\Pf_{m-1}(\varphi)=I_G$ and $3$ of the submaximal pfaffians of $\varphi$ are exactly  $p_1,p_2,p_3.$
\end{lem}
\begin{proof}
Take any alternating map $\psi: F^{\vee} \to F$ of rank odd $n$ such that $\Pf_{n-1}(\psi)=I_G=(g_1, \ldots g_n),$ then apply the previous observation to get a new alternating map  $\varphi: H^{\vee} \to H$ of odd rank $m=n+6$ such that $\Pf_{m-1}(\varphi)=I_G=(g_1, \ldots g_n,p_1,p_2,p_3, 0,0,0).$ 
\end{proof}

We need also the following lemma.

\begin{lem}
Let $M$ be an alternating matrix of odd size $m,$ with entries in a unitary commutative ring $R.$ Let $I$ be the ideal generated by the submaximal pfaffians of $M$ and let $(q_1,\ldots,q_m)$ be a set of generators of $I.$ Then there exists an alternating matrix whose submaximal pfaffians are $uq_1,\ldots,uq_m,$ where $u$ is a unit of $R.$ 
\end{lem}
\begin{proof}
Let ${\bf p}=(p_1\ldots p_m)$ be the vector of the submaximal pfaffians of $M$ and let ${\bf q}=(q_1\ldots q_m).$ Then ${\bf p}={\bf q}A,$ where $A$ is an invertible square matrix of size $m.$ 

Let us consider the alternating matrix $AM\:{}^t\!\!A.$  Now, if $N$ is an alternating matrix we write $\pf N$ for the vector of the submaximal pfaffians of $N$ and if $A$ is a square matrix we write $\overline{A}$ for the adjoint matrix of $A.$ A straightforward computation shows that
 $$\pf(AM\:{}^t\!\!A)=(\pf M)\overline{A}={\bf p}\overline{A}=(\det A){\bf q}.$$
Since $A$ is invertible, $\det A$ is a unit, so we are done. 
\end{proof}

Now recall that if $A_G=R/I_G$ is a Gorenstein graded algebra of depth $3$ ($R=k[x_0, \ldots, x_r]$), its graded minimal free resolution is of the type
$$0\to   R(-\vartheta_G) \to \bigoplus_{h \in \vartheta_G-H}R(-h) \to \bigoplus_{h \in H}R(-h) \to R \to A_G \to 0$$
where $H$ is the multiset of the degrees of the minimal generators of $I_G$ and $\vartheta_G=2\|H\|/ (|H|-1)$ (remind that $|H|$ must be odd). Hence the multiset $H$ determines all the graded Betti numbers of $A_G.$ Moreover, the integers in $H=\{\{h_1\le h_2 \le  \ldots \le h_{2m+1}\}\}$ must satisfy the following Gaeta-Diesel conditions (see \cite{Ga} and \cite{Di}): $ \vartheta_G$ is an integer and $\vartheta_G > h_{i+1} + h_{2m+2-i}$ for $i=1, \ldots, m.$
\par
In general when $T$ is a graded $R$-module, ($R$ any polynomial ring), and 
\begin{multline*} $$F_{\bullet} \quad 0\to\bigoplus_{h\in M_p}R(-h)\to\ldots\to\bigoplus_{h\in M_{i+1}}R(-h)\to \bigoplus_{h\in M_{i}}R(-h)\to\ldots \\ \ldots \to \bigoplus_{h\in M_1}R(-h)$$\end{multline*}
is a graded  free resolution of $T$ with each $M_j$ a multiset of positive integers, we say $(M_1, \ldots, M_p)$ the {\em multiset sequence} associated to the resolution  $F_{\bullet}.$ 

If $s\in \Supp M_i\cap \Supp M_{i+1},$  we say that $s$ is a {\em repetition} in the resolution. Moreover, if $N \subseteq M_i\cap M_{i+1}$ we say that $N$ is {\em cancellable} for $T$ if there is a graded free resolution of $T$ whose associated multiset sequence is $(M_1, \ldots,M_{i}\setminus N,M_{i+1}\setminus N, \ldots, M_p).$
 
Next proposition collects some numerical facts about the graded free resolutions of $3$-codimensional standard graded Gorenstein algebras.
\begin{prp}
Let 
 $$0\to\bigoplus_{h\in C}R(-h)\to\bigoplus_{h\in B}R(-h)\to\bigoplus_{h\in A}R(-h)\to R\to R/I_G\to 0$$
be a graded free resolution (not necessarily minimal)  of a depth $3$ Gorenstein graded algebra. 
\begin{itemize}
	\item[i.] $|B\cap C|=|C|-1.$ 
	\item[ii.] Let $\vartheta$ be the unique element in $\Supp (C\setminus B).$ If $s\in \Supp(A\cap B)$ is such that $\mu_{A\cap B}(s)> \mu_{A\cap B}(\vartheta-s)$ then the multiset $\{s^{(\mu_{A\cap B}(s)-\mu_{A\cap B}(\vartheta-s))}\}$ is cancellable for $R/I_G.$ 
\end{itemize}
\end{prp}
\begin{proof}
i. It is enough to remind that in a minimal free resolution of a Gorenstein algebra the module of last syzygies has rank one.\\
ii. It is enough to remind that in a minimal free resolution of a depth $3$ Gorenstein algebra, by selfduality,  $\mu_{A \cap B}(s)=\mu_{A \cap B}(\vartheta-s).$
\end{proof}
Since we need  complete intersection ideals contained in Gorenstein ideals, we state some results from the paper~\cite{RZ2}. 
 
Let 
$$\beta=\big(\{\{d_1,\ldots, d_{2n+1}\}\},\{\{\vartheta-d_{2n+1},\ldots,\vartheta-d_1\}\},\{\vartheta\}\big),$$
$n\vartheta=\sum d_i,$ $d_1\le\ldots\le d_{2n+1},$ be a Betti sequence admissible for an depth $3$  standard graded Gorenstein algebra. 
Let
 $$\CI^g_{\beta}=\{\mathbf{a}\in\Nm^3\mid\exists\mbox{ an ideal }I\subset R \mbox{ containing a regular sequence of type }\mathbf{a}$$ $$\mbox{ with }{\beta}_{R/I}=\beta, \, R/I \mbox { Gorenstein ring}\},$$
where $$\Nm^3=\{(a_1,a_2,a_3)\in{\mathbb N}^3\mid a_1\le a_2\le a_3\}.$$
 
In ~\cite{RZ2} Theorem 3.6 it was shown that the poset $\CI^g_{\beta}$ has only one minimal element  and it was computed. We report here that statement. 

\begin{thm}\label{c3b}
Let $\beta=\big(\{\{d_1,\ldots, d_{2n+1}\}\},\{\{\vartheta-d_{2n+1},\ldots,\vartheta-d_1\}\},\{\vartheta\}\big),$
$n\vartheta=\sum d_i,$ $d_1\le\ldots\le d_{2n+1},$ be a Betti sequence admissible for an Artinian Gorenstein quotient of $k[x_1,x_2,x_3]$ and define the sets
$$B=\{3\le i\le n+1\mid\vartheta\le d_i+d_{2n+4-i}\}$$
and
 $$C=\{4\le i\le n+2\mid\vartheta\le d_i+d_{2n+5-i}\}.$$
Then $\CI^g_{\beta}$ has a unique minimal element which we will call $\mci\beta.$
Precisely, 
\begin{itemize}
\item [i)] if $B \ne \emptyset,$ then $\mci\beta=(d_1,d_{\max B}, d_{2n+4-\min B});$
\item [ii)] if $B = \emptyset$ and $C \ne \emptyset,$ then $\mci\beta=(d_1,d_2, d_{\max C});$
\item [iii)] if $B = \emptyset$ and $C = \emptyset,$ then $\mci\beta=(d_1,d_2, d_3).$
\end{itemize}
In particular $\CI^g_{\beta}=\{\mathbf{b}\in\Nm^3\mid\mathbf{b}\ge\mci\beta\}.$ 
\end{thm}
In the sequel we will need also the following proposition which is a reformulation of Lemma 3.7 of ~\cite{RZ2}.

\begin{prp}\label{BC}
Let $H$ be an admissible Hilbert function for an Artinian Gorenstein quotient of $k[x_1,x_2,x_3]$ and let $\beta=(D,\vartheta -D,\{\vartheta\})$ be a Gorenstein Betti sequence compatible with $H,$  with $D= \{\{d_1,\ldots, _{2n+1}\}\},$ $d_1 \le \ldots \le d_{2n+1}.$ Let $C$ be as in Theorem~\ref{c3b} and 
$$\overline{B}:=\{3\le i\le 2n+1\mid\vartheta\le d_i+d_{2n+4-i}\}.$$ We have
\begin{itemize}
  \item[a)] $i\in\overline{B}$ $\rw$ $\mu_D(d_i)=-\Delta^2 H(d_i).$
  \item[b)] $i\in C,$ $i\not\in\overline{B},$ $i-1\not\in\overline{B}$ $\rw$ $\mu_D(d_i)=-\Delta^2 H(d_i)-1.$
  \item[c)] $i,j\in\overline{B},$ $d_i=d_j$ $\rw$ $i=j.$
  \item[d)]If $\mu_D(d_i)=-\Delta^2 H(d_i)$ and $k=\min\{j\,|\ d_j=d_i\}$ then $k\in\overline{B}.$
  \item[e)]If $\mu_D(d_i)=-\Delta^2 H(d_i)-1$ and $k=\min\{j\,|\ d_j=d_i\}\le n+2$ then $k\in C.$
\end{itemize}
\end{prp}
\begin{proof}
We set $\nu_h:=\mu_D(h)$ and $\sigma_h:=\mu_{\vartheta -D}(h).$
\begin{itemize}
    \item[a)]Let $i\in\overline{B};$ then $d_i\ge\vartheta-d_{2n+4-i}>d_{i-1};$ 
        furthermore $\vartheta-d_{2n+3-i}>d_i,$ so 
    $$-\Delta^2 H(d_i)=\sum_{h=1}^{d_i}\nu_{h}-\sum_{h=1}^{d_i}\sigma_{h}-1=(\nu_{d_i}+i-1)-(i-2)-1=\nu_{d_i}.$$           \item[b)]Let $i\in C\setminus\overline{B}$ such that $i-1\not\in\overline{B};$ then $d_i\ge\vartheta-d_{2n+5-i}>d_{i-1};$ furthermore, since $i\not\in\overline{B}$, $\vartheta-d_{2n+4-i}>d_i,$ so 
    $$-\Delta^2 H(d_i)=\sum_{h=1}^{d_i}\nu_{h}-\sum_{h=1}^{d_i}\sigma_{h}-1=(\nu_{d_i}+i-1)-(i-3)-1=\nu_{d_i}+1.$$  
    \item[c)]If $i<j$ then $d_i=d_{i+1}$ so we should have
    $$\vartheta\le d_i+d_{2n+4-i}=d_{i+1}+d_{2n+4-i}<\vartheta.$$
    \item[d)]We have $d_i>d_1,$ therefore $k>1$ and $d_k>d_{k-1};$ consequently $$\sum_{h=1}^{d_i-1}\nu_{h}=\sum_{h=1}^{d_k-1}\nu_{h}=k-1;$$
    moreover by hypotheses
    $$-\Delta^2 H(d_i)=\sum_{h=1}^{d_i}\nu_{h}-\sum_{h=1}^{d_i}\sigma_{h}-1=\nu_{d_i}\rw
    \sum_{h=1}^{d_k-1}\nu_{h}-\sum_{h=1}^{d_k}\sigma_{h}=1$$
    i.e.
    $$\sum_{h=1}^{d_k}\sigma_{h}=k-2\rw\vartheta-d_{2n+4-k}\le d_k\rw k\in\overline{B}.$$
    \item[e)]Again $d_i>d_1,$ therefore $k>1$ and $d_k>d_{k-1};$ consequently $$\sum_{h=1}^{d_i-1}\nu_{h}=\sum_{h=1}^{d_k-1}\nu_{h}=k-1;$$
    moreover by hypotheses
    $$-\Delta^2 H(d_i)=\sum_{h=1}^{d_i}\nu_{h}-\sum_{h=1}^{d_i}\sigma_{h}-1=\nu_{d_i}+1\rw
    \sum_{h=1}^{d_k-1}\nu_{h}-\sum_{h=1}^{d_k}\sigma_{h}=2$$
    i.e.
    $$\sum_{h=1}^{d_k}\sigma_{h}=k-3\rw\vartheta-d_{2n+5-k}\le d_k,$$
    hence $k\in C.$ 
\end{itemize}
\end{proof}

\begin{rem}\label{maxg}
We remind that by~\cite{RZ1}, Proposition 3.7, for every $d_i>d_1,$ $-\Delta^2 H(d_i)$ is equal to the largest number of minimal generators
of degree $d_i$ (which we will denote by $\Mng_H(d_i)$) compatible with an Artinian Gorenstein graded algebra of codimension $3$ having Hilbert function equal to $H.$
\end{rem}

\begin{rem}\label{Bvuoto}
Note that if $B=\emptyset$ then $\overline{B}\subseteq\{n+2\}.$ Indeed let $i\in\overline{B};$ since $B=\emptyset$ we have that $i\ge n+2;$ moreover $\vartheta\le d_i+d_{2n+4-i},$ so $2n+4-i\ge n+2,$ hence $i\le n+2$ i.e. $i=n+2.$
\end{rem}


\section{Structure theorem for almost complete intersections} 
\markboth{\it Structure theorem for almost complete intersections}
{\it Structure theorem for almost complete intersections}
In the sequel, if $f:M \to N$ and $g:M \to P$ are maps of modules we will write $(f,g): M \to N\oplus P$ for the map defined by $(f,g)(m)=(f(m),g(m)).$ Moreover, if 
$f:M \to P$ and $g:N \to P$ we will write $f|g: M\oplus N \to P$ for the map defined by $f|g(m,n)=f(m)+g(n).$
\par
Let $R$ be a Noetherian local commutative ring, $H_0$ a free $R$-module of odd rank $m_0\ge 5$ and $\varphi_0:H_0^{\vee}\to H_0$ an alternating map. Let $J:=\Pf_{m_0-1}(\varphi_0)$ be the ideal generated by the pfaffians of $\varphi_0$ of size $m_0-1$ (note that despite the fact that the pfaffians depend on the choice of the bases in $H_0,$ $J$ depends only on the map $\varphi_0$). Moreover, we denote by $\pf \varphi_0: H_0 \to R$ the map defined by the submaximal pfaffians of $\varphi_0.$ It is known by Buchsbaum-Eisenbud Theorem (see \cite{BE}) that $\depth J\le 3$ and we suppose here that $\depth J=3.$ Take a regular sequence $(p_1,p_2,p_3)$ in $J.$ By Lemma \ref{3gen} there exists an alternating map $\varphi: H^{\vee} \to H$ of odd rank $m$ such that $\Pf_{m-1}(\varphi)=J$ and $3$ of the submaximal pfaffians of $\varphi$ are exactly  $p_1,p_2,p_3.$ So $J=(p_1,p_2,p_3, \ldots, p_m),$ where the $p_i$'s, for $1\le i\le m,$ are the submaximal pfaffians of $\varphi.$
\par
Let $I=(p_1,p_2,p_3)$ and $\psi:G^{\vee}\to G$ an alternating map whose pfaffians are exactly $p_1,p_2,p_3.$ Then $H=G\oplus F,$ where $F$ is a free $R$-module of rank $m-3,$ and $\pf\varphi=\pf\psi|\sigma,$ where $\sigma:F\to R.$
Therefore we have the following decomposition: $\varphi=\big(\alpha\,|-\lambda^{\vee},\lambda|\beta\big),$ where $\alpha:G^{\vee}\to G,$ $\beta:F^{\vee}\to F$ and $\lambda:G^{\vee}\to F.$ Moreover we denote by 
$\overline{\beta}:F\to F^{\vee},$ the alternating map such that $\beta\overline{\beta}=p\idt$ and $\overline{\beta}\beta=p\idt,$ where $p$ is the pfaffian of the map $\beta.$ One can see that (with respect to suitable bases) the matrix associated to $\overline{\beta}$ is the pfaffian adjoint of the matrix of $\beta.$

[The pfaffian of an empty matrix will be $1$; note that  $\pf\psi^{\vee}=-(\pf\psi)$]

\begin{prp}\label{constr} With the above notation
$$\xymatrix@R=1.5cm@C=1.3cm{
0\ar[r]&F^{\vee}\ar[r]^(.4){(\lambda^{\vee},-\beta)}&G\oplus F\ar[rr]^{\big(p\, | \lambda^{\vee}\overline\beta,-\pf\psi\, |- \sigma\big)}&&G\oplus R\ar[r]^-{\pf\psi | p}&R}$$
is a free resolution of an  almost complete intersection algebra $R/Q.$
\end{prp} 
\begin{proof}
We start by proving that the canonical surjection $R/I\to R/J$ can be lifted to a map between their resolutions in the following way:
 
 $$\xymatrix@R=1.5cm@C=1.3cm{
0\ar[r]&R\ar[r]^{(\pf\psi)^{\vee}}\ar[d]^{p}&{\phantom{a}G^{\vee}}\ar[rr]^{\psi}\ar[d]^{(p,-\overline\beta\lambda)}&&
	G\ar[r]^{\pf\psi}\ar[d]^{(\idt,0)}&R\ar[d]^{\idt} \\
	0\ar[r]&R\ar[r]^(.4){((\pf\psi)^{\vee},\sigma^{\vee})}&G^{\vee}\oplus F^{\vee}\ar[rr]^-{\big(\alpha\,|-\lambda^{\vee},\lambda|\beta\big)}&&G\oplus F\ar[r]^-{\pf\psi | \sigma}&R
	 }$$
Indeed, since
	$$0=(\lambda|\beta)((\pf\psi)^{\vee},\sigma^{\vee})=\lambda(\pf\psi)^{\vee}+\beta\sigma^{\vee}\rw$$
	$$\rw\overline{\beta}\lambda(\pf\psi)^{\vee}+\overline{\beta}\beta\sigma^{\vee}=0\rw
	\overline{\beta}\lambda(\pf\psi)^{\vee}+p\sigma^{\vee}=0$$
we have
$$(p,-\overline\beta\lambda)(\pf\psi)^{\vee}=(p(\pf\psi)^{\vee},-\overline\beta\lambda(\pf\psi)^{\vee})
	=(p(\pf\psi)^{\vee},p\sigma^{\vee})=p((\pf\psi)^{\vee},\sigma^{\vee}).$$
	
	 Now Remark \ref{lpl} will imply, in our notation, that $\alpha p+\lambda^{\vee}\overline\beta\lambda=\psi.$  Therefore we have
\begin{multline*}
	\big(\alpha\,|-\lambda^{\vee},\lambda|\beta\big)(p,-\overline\beta\lambda)=\\
	=(\alpha p+\lambda^{\vee}\overline\beta\lambda,\lambda p-\beta\overline\beta\lambda)=
	(\psi,\lambda p-p\lambda)=(\psi,0)=(\idt,0)\psi.
\end{multline*}

Then we set ${\mathbb F}^I_{\bullet}$ and ${\mathbb F}^J_{\bullet}$ the above resolutions of $R/I$ and $R/J$ and $\tau:{\mathbb F}^I_{\bullet}\to{\mathbb F}^J_{\bullet}$ the above complex map. Thus, if we set $Q:=I:J,$ we see that a free resolution of $R/Q$ is given by the mapping cone of the map $\tau^{\vee}:({\mathbb F}^J_{\bullet})^{\vee}\to
({\mathbb F}^I_{\bullet})^{\vee}.$ If we make cancellations where we have identities map we get the required resolution of the almost complete intersection $R/Q.$
\end{proof}
As a simple consequence of the previous result we have 

\begin{cor}\label{gen}
Let $R$ be a Noetherian local ring and let $M$ be an alternating matrix of odd rank $m,$ whose entries are in $R.$ Let $p_1,\ldots,p_m$ be the submaximal pfaffians of $M$ and $p_{abc}$ the pfaffian of order $m-3$ obtained by $M$ by deleting the rows and columns $a,b,c$. If $(p_a,p_b,p_c)$ is a regular sequence, then
 $$(p_a,p_b,p_c):(p_1,\ldots,p_m)=(p_a,p_b,p_c,p_{abc}).$$ 
\end{cor}


 The Proposition \ref{constr} gives a way to construct almost complete intersection algebras but the very interesting  thing is that every almost complete intersection can be constructed in this way.
 
 \begin{thm}\label{tqci}
Let $R$ be a Noetherian local ring and $I_Q\subset R$ be a perfect ideal of depth $3$ of an almost complete intersection. Then there exists an alternating map $\psi:H^{\vee}\to H,$ where $H$ is a free $R$-module of odd rank with $\im\pf\psi$ of depth $3$, such that, with the above notation, 
 $$\xymatrix@R=1.5cm@C=1.3cm{
0\ar[r]&F^{\vee}\ar[r]^(.4){(\lambda^{\vee},-\beta)}&G\oplus F\ar[rr]^{\big(p\, | \lambda^{\vee}\overline\beta,-\pf\psi\, |- \sigma\big)}&&G\oplus R\ar[r]^-{\pf\psi | p}&R}$$
is a free resolution of $R/I_Q.$
\end{thm} 
\begin{proof}
Since $I_Q$ is an almost complete intersection ideal we have that $I_Q=(p_0,p_1,p_2,p_3)$ and since $\depth I_Q=3$ we can suppose that $I_Z:=(p_1,p_2,p_3)$ is generated by a regular sequence. Let $I_G:=I_Z:I_Q.$ Then $I_G$ is a Gorenstein ideal of depth $3.$ By Lemma \ref{3gen}, there exists an alternating map $\varphi:H^{\vee}\to H,$ such that the submaximal pfaffians of $\varphi$ are exactly $p_1,p_2,p_3$ and the other generators of $I_G.$ By Proposition~\ref{constr} the above resolution gives a resolution of $R/I_Q.$
\end{proof}

Now we would like to give a graded version of the previous result.

Let $\pp^r_k$ be the projective space with $r \ge 3$ and $R=k[x_0,x_1,\ldots,x_r]$ the standard graded coordinate ring of $\pp^r_k.$ 

\begin{thm}\label{grqci}
Let $Q \subset \pp^r_k$ be an almost complete intersection scheme of codimension $3$ and $I_Q\subset R$ be its defining ideal. Then $R/I_Q$ admits a graded free resolution of the following type 
 $$0\to K^{\vee}(-d) \to G(-d_0) \oplus K \to G \oplus R(-d_0)\to R$$
 where $d_0$ is a positive integer, $G=\oplus_{i=1}^3R(-d_i),$ $d=d_o+d_1+d_2+d_3,$ $K=\oplus_{i=4}^mR(-e_i),$ $m\ge 5$ an odd integer and  $d_1, d_2, d_3,e_4-d_0, \ldots, e_m-d_0$ are the degrees of all submaximal pfaffians of a suitable alternating matrix of size $m.$
\end{thm}
\begin{proof} 
Let $Z\subset\pp^r_k$ be a complete intersection generated by $3$ minimal generators of $I_Q,$ of degrees $d_1,d_2,d_3.$ Let $I_\Gamma:=I_Z:I_Q.$ $I_\Gamma$ is the saturated homogeneous ideal of an aG scheme $\Gamma$ directly linked to $Q$ in $Z.$ As in Proposition~\ref{constr} we have the following diagram
 $$\xymatrix@R=1.3cm@C=1.1cm{
0\ar[r]&R(-\vartheta_Z)\ar[r]\ar[d]&{\phantom{a}G^{\vee}}(-\vartheta_Z)\ar[r]\ar[d]&
	G\ar[r]\ar[d]&R\ar[d] \\
	0\ar[r]&R(-\vartheta_G)\ar[r]&G^{\vee}(-\vartheta_G)\oplus F^{\vee}(-\vartheta_G)\ar[r]^{\phantom{aaaaaaa}\varphi}&G\oplus F\ar[r]&R
	 }$$
where the first row is a minimal graded resolution of $R/I_Z$ and the second row is the graded resolution of $R/I_{\Gamma},$ obtained by the alternating map $\varphi$ as in Lemma~\ref{3gen}. So we have that $G=\oplus_{i=1}^3R(-d_i)$ and $F=\oplus_{i=4}^mR(-d_i).$ Dualizing and shifting by $-\vartheta_Z$ this diagram we get
  $$\xymatrix@R=1.1cm@C=0.5cm{
0\ar[r]&R(-\vartheta_Z)\ar[r]\ar[d]&G^{\vee}(-\vartheta_Z)\oplus F^{\vee}(-\vartheta_Z)\ar[d]\ar[r]&G(-d_0)\oplus F(-d_0)\ar[r]\ar[d]&R(-d_0)\ar[d] \\
	0\ar[r]&R(-\vartheta_Z)\ar[r]&{\phantom{a}G^{\vee}}(-\vartheta_Z)\ar[r]&
	G\ar[r]&R 
	 }$$
where $d_0:=\vartheta_Z-\vartheta_G.$ Taking the mapping cone and after the trivial cancellations we obtain the following resolution of $R/I_Q$
  $$0\to F^{\vee}(-\vartheta_Z)\to G(-d_0)\oplus F(-d_0)\to G\oplus R(-d_0)\to R.$$
Now if we set $K:=F(-d_0)$ and $d:=2\vartheta_Z-\vartheta_G=d_0+d_1+d_2+d_3$ we get
  $$0\to K^{\vee}(-d) \to G(-d_0) \oplus K \to G \oplus R(-d_0)\to R,$$
which is the required resolution.
\end{proof}

\begin{cor}\label{d0}
Let $Q \subset \pp^r_k$ be an almost complete intersection scheme of codimension $3$ and $I_Q\subset R$ be its defining ideal, then $R/I_Q$ admits a graded free resolution of the type 
$$0\to K^{\vee}(-d) \to G(-d_0) \oplus K \to G \oplus R(-d_0)\to R$$
in which $d_0=\min\{\deg p \ | \ p\in I_Q\}.$
\end{cor}
\begin{proof}
Let $d_0\le d_1\le d_2 \le d_3$ the degrees of a minimal set of generatorso of $I_Q.$ Since we can find a regular sequence of minimal generators in $I_Q$ of type $d_1, d_2,d_3,$ the conclusion follows by repeating the argument of Theorem \ref{grqci} using such a regular requence. 
\end{proof}

\begin{exm}
Let $Q\subset\pp^3_k$ be the $0$-dimensional almost complete intersection linked to $5$ general points ($\Gamma$) in a complete intersection $Z$ of type $(2,2,8).$ Then we consider the following graded resolutions
 $$0\to R(-12)\to R(-4)\oplus R(-10)^2\to R(-2)^2\oplus R(-8)\to R\to R/I_Z\to 0$$
and
\begin{multline*}
 0\to R(-5)\to R(-8)\oplus R(-3)^5\oplus R(3)\to R(-8)\oplus R(-2)^5\oplus R(3)\to \\ \to R\to R/I_{\Gamma}\to 0,
\end{multline*}
where the last one is the pfaffian resolution of $R/I_{\Gamma}$ got by the minimal one by adding the term 
$R(-8)\oplus R(3)$ to the second and to the third module of the complex. Such a resolution can be built as illustrated in Lemma~\ref{3gen} and the pfaffians of its alternating central map are the five minimal generators of $I_{\Gamma}$ of degree $2,$ the form of $I_{\Gamma}$ of degree $8$ used to perform the linkage and the null form.
So we have $G=R(-2)^2\oplus R(-8),$ $K=R(-9)^3\oplus R(-4),$ $d_0=7$ and $d=19.$
Consequently we get the following resolution of $R/I_Q$
\begin{multline*}
 0\to R(-15)\oplus R(-10)^3\to [R(-9)^2\oplus R(-15)]\oplus[R(-9)^3\oplus R(-4)]\to \\ 
 \to[R(-8)\oplus R(-2)^2]\oplus R(-7)\to R\to R/I_{Q}\to 0.
\end{multline*}
Note that a minimal graded resolution of $R/I_Q$ can be obtained by deleting the term $R(-15).$
\end{exm}

\section{The graded Betti numbers of almost complete intersections} 
\markboth{\it The graded Betti numbers of almost complete intersections}
{\it The graded Betti numbers of almost complete intersections}
In this section we would like to characterize all graded Betti sequences admissible for an almost complete intersection scheme of codimension $3$ of $\pp^r.$
\par
Since we are interested to the graded Betti numbers for such schemes we can restrict ourselves to the Artinian reduction of $X.$ So, from now on, we let $R=k[x_1,x_2,x_3]$ and $A_Q= R/I_Q$ an Artinian almost complete intersection graded algebra. A graded minimal free resolution of $A_Q$ is of the following form 
$$0\to   \bigoplus_{h \in F}R(-h) \to \bigoplus_{h \in E}R(-h) \to \bigoplus_{h \in D}R(-h) \to R \to A_Q \to 0$$
where $D,E,F$ are multisets of positive integers with $|D|=4,$ $|E|=|F|+3.$

For this aim we will consider the following multisets of positive integers
$D=\{\{d_0\le d_1\le d_2\le d_3\}\},$ $E=\{\{e_i\}\}_{i=1}^{p+3}$ and $F=\{\{f_i\}\}_{i=1}^{p}$ where $p\ge 2.$ In the sequel we will denote $D^*=\{\{ d_1\le d_2\le d_3\}\}.$

\begin{lem}\label{l1}
Let $(D,E,F)$ be a Betti sequence admissible for an Artinian almost complete intersection graded algebra of codimension $3.$ Then 
\begin{itemize}
\item [1)] $d-F \subset E,$ where $d=\|D\|;$
\item [2)] if we set $\widehat{E}:=E\setminus (d-F),$ $S:=D^*\cap (\vartheta_Z-\widehat{E}),$ $\overline{D}:=D^* \setminus S$ then $\widehat{E}=(d_0+\overline{D})\sqcup (\vartheta_Z-S).$ 
\end{itemize}
\end{lem}

\begin{proof} 1) Let $A=R/I_Q $ be an Artinian almost complete intersection graded algebra such that $\beta_{A}=(D,E,F).$ $I_Q$ contains a length $3$ regular sequence of type $(d_1,d_2,d_3)$ which is part of a minimal set of generators for $I_Q;$ let $I_Z$ be the ideal generated by such a regular sequence. Then $I_G:=I_Z:I_Q$  is a Gorenstein ideal of depth $3.$ By standard mapping cone procedure we obtain the following  graded free resolution of $R/I_G$
$$0\to R(-\vartheta_G)\to\bigoplus_{h\in \vartheta_Z-E}R(-h)\to \bigoplus_{h\in \vartheta_Z-F }R(-h)\oplus \bigoplus_{h\in D^*}R(-h)\to R$$
where $D^*=\{\{d_1,d_2,d_3\}\},$ $\vartheta_Z=\|D^*\|,$ $\vartheta_G=\vartheta_Z-d_0.$ Since, by mapping cone, no summand of $\bigoplus_{h\in \vartheta_Z-F }R(-h)$ is cancellable in such a resolution, by Gorenstein duality, $\vartheta_G-(\vartheta_Z-F) \subset \vartheta_Z-E$  hence $\vartheta_Z-(\vartheta_G-(\vartheta_Z-F)) \subset \vartheta_Z-(\vartheta_Z-E),$ i.e. $d-F \subset E.$

2) By our setting we have $D^*=S \sqcup \overline{D}$ and $\vartheta_Z-\widehat{E}=S \sqcup [(\vartheta_Z-\widehat{E})\setminus S];$ therefore in the previous resolution we have that 
$$\bigoplus_{h\in \vartheta_Z-E}R(-h)= \bigoplus_{h\in S}R(-h)\oplus\bigoplus_{h\in (\vartheta_Z-\widehat{E})\setminus S}R(-h)\oplus \bigoplus_{h\in -d_0+F }R(-h)$$
$$\bigoplus_{h\in D^*}R(-h)=\bigoplus_{h\in S}R(-h)\oplus \bigoplus_{h\in \overline{D}}R(-h)$$
where the only summands eventually cancellable are in $\bigoplus_{h\in S}R(-h),$ therefore, by Gorenstein duality, $\vartheta_G-\overline{D}=(\vartheta_Z-\widehat{E})\setminus S,$ hence $\vartheta_Z-(\vartheta_G-\overline{D})= \vartheta_Z-[(\vartheta_Z-\widehat{E})\setminus S],$ i.e. $d_0+\overline{D}=\widehat{E}\setminus (\vartheta_Z -S)$ and we are done. 
\end{proof}

Because of the previous lemma it is convenient to give the following definition.

\begin{dfn}\label{aci}
A sequence $(D,E,F)$ of multisets of positive integers is said of \aci-type if it satisfies the following conditions
\begin{itemize}
	\item[1)]$|D|=4,$ $|E|=|F|+3,$ $|F|\ge 2;$
	\item[2)]$d-F\subset E,$ where $d:=\|D\|;$
	\item[3)]if we set $\widehat{E}:=E\setminus (d-F),$ $d_0:=\min D,$ $D^*:=D\setminus\{d_0\},$ $\vartheta_Z=\|D^*\|,$ $S:=D^*\cap (\vartheta_Z-\widehat{E}),$ $\overline{D}:=D^* \setminus S$ then $\widehat{E}=(d_0+\overline{D})\sqcup (\vartheta_Z-S).$ 
\end{itemize}
\end{dfn}

\begin{rem}\label{minres}
Let $A=R/I_Q $ be an Artinian almost complete intersection graded algebra such that $\beta_{A}=(D,E,F);$  let $I_Z$ be an ideal generated by a regular sequence of type $(d_1,d_2,d_3)$ which is part of a minimal set of generators for $I_Q$ and  $I_G:=I_Z:I_Q$  the linked Gorenstein ideal of depth $3.$ Then the minimal graded free resolution of $A_G=R/I_G$ has the following form
$$
\begin{array}{c}
\displaystyle{\bigoplus_{h\in \vartheta_Z-F }R(-h)\oplus \bigoplus_{h\in \overline{D}}R(-h)	\oplus \bigoplus_{h\in \overline{S}}R(-h)}\\
\uparrow\\
\displaystyle{\bigoplus_{h\in -d_0+F }R(-h)\oplus \bigoplus_{h\in\vartheta_G- \overline{D}}R(-h)	\oplus \bigoplus_{h\in \overline{S}}R(-h)}\\
\uparrow\\
R(-\vartheta_G)\\
\uparrow\\
0
\end{array}
$$
where $\overline{S}\subseteq S.$ Now observe that $\vartheta_G- [(\vartheta_Z-F)\sqcup \overline{D}]=(-d_0+F) \sqcup (\vartheta_G- \overline{D})$, so we can apply Proposition \ref{dual} and deduce that for  $\overline{S}$ there are $4$ possibilities depending on its cardinality.
\begin{itemize}
	\item [0)] $\overline{S}=\emptyset;$
	\item [1)] $|\overline{S}|=1$ in such a case, by Proposition 3.7 and Remark 3.8 in \cite{RZ1}, $\overline{S}=\{\vartheta_G/2\};$
	\item [2)] $|\overline{S}|=2$ in such a case, by Proposition 3.7 and Remark 3.8 \cite{RZ1}, $\overline{S}=\{\{\alpha, \vartheta_G-\alpha\}\}$ for some $\alpha;$
	\item [3)] $|\overline{S}|=3$ in such a case, by Proposition 3.7 and Remark 3.8 \cite{RZ1}, $\overline{S}=\{\{\vartheta_G/2,\alpha, \vartheta_G-\alpha\}\}$ for some $\alpha.$
\end{itemize}
In particular, if $\vartheta_G/2 \notin \Supp S$ then $|\overline{D}|+|F|$ is odd and consequently $|S|+|F|$ is even. It is easy to produce examples in which  $\vartheta_G/2 \in \Supp S$ and $|S|+|F|$ is even and examples in which $\vartheta_G/2 \in \Supp S$ and $|S|+|F|$ is odd.
\end{rem}

Now we prove two technical lemmas which will be crucial for the characterization of the Betti sequences of almost complete intersections.

\begin{lem}\label{magg}
Let $A_G$ be an Artinian Gorenstein graded algebra of codimension $3$ whose Betti sequence is $\beta_G.$ Let 
$(d_1,d_2,d_3)\ge(e_1,e_2,e_3):=\mci\beta_G,$ with $d_1\le d_2\le d_3$ and $e_1\le e_2\le e_3.$
\begin{itemize}
	\item[1)]If $d_i=e_i$ for some $i$ then for every regular sequence $(g_1,g_2,g_3)$ in $I_G,$ with $\deg g_j=d_j$ for $1\le j\le 3,$ $g_i$ is a minimal generator for $I_G.$ 
	\item[2)]If $d_i>e_i$ for some $i$ and $I_G$ has a minimal generator of degree $d_i$ then there are in $I_G$ regular sequences $(g_1,g_2,g_3),$ with $\deg g_j=d_j$ for $1\le j\le 3,$ such that $g_i$ is a minimal generator for $I_G$ and regular sequences $(h_1,h_2,h_3),$ with $\deg h_j=d_j$ for $1\le j\le 3,$ such that $h_i$ is not a minimal generator for $I_G.$
\end{itemize}
\end{lem}
\begin{proof}
1) Let us suppose that $d_i=e_i$ for some $i$ and let us consider a regular sequence $(g_1,g_2,g_3)$ in $I_G,$ with $\deg g_j=d_j$ for $1\le j\le 3.$ It is enough to remind that if $i=1$ then $d_1=\min\{\deg f\mid f\in I_G\};$ if $i=2$ then $\depth(I_G)_{\le d_2-1}=1;$ if $i=3$ then $\depth(I_G)_{\le d_3-1}=2.$

2) Let us suppose that $d_i>e_i$ for some $i$ and that $I_G$ has a minimal generator of degree $d_i.$ Let 
	$(f_1,f_2,f_3)$ be a regular sequence in $I_G$ such that $\deg f_j=e_j$ for $1\le j\le 3.$ If we choice $h_j:=f_ja_j,$ where $a_j$ is a generic form of degree $d_j-e_j,$ for $1\le j\le 3,$ we get a regular sequence in which $h_i$ is not a minimal generator for $I_G.$ Now we set $g_j:=f_ja_j,$ where $a_j$ is a generic form of degree $d_j-e_j$ for $j\ne i$ and we take as $g_i$ a generic form in $(I_G)_{d_i}.$ Of course, $(g_1,g_2,g_3)$ is a regular sequence and, since by hypothesis $I_G$ has minimal generators in degree $d_i,$ $g_i$ will be a minimal generator for $I_G.$ 
\end{proof}

\begin{lem}\label{zzz}
Let $A_{\Gamma}=R/I_{\Gamma}$ be an Artinian Gorenstein graded algebra of codimension $3,$ whose last syzygy degree is $\vartheta_{\Gamma}.$ Let $(f_1,f_2,f_3)$ be a regular sequence in $I_{\Gamma},$ with $\deg f_i=d_i$ and $d_1\le d_2\le d_3.$ Let us suppose that $d_i+d_j=\vartheta_{\Gamma}$ and  $f_i,$ $f_j$ are minimal generators for $I_{\Gamma}$ for some $1 \le i<j\le 3.$ Let $A_{G}$ be an Artinian Gorenstein graded algebra of codimension $3$ whose Betti sequence is obtained from the Betti sequence of $A_{\Gamma}$ by deleting the degrees $d_i$ and $d_j$ between generators and first syzygies. If $(e_1,e_2,e_3):=\mci\beta_G,$ then $d_i>e_i$ and $d_j>e_j.$
\end{lem}
\begin{proof}
Along this proof we let $\gamma_1\le\ldots\le\gamma_{2m+1}$ be the degrees of a minimal set of generators for $I_{\Gamma},$ $g_1\le\ldots\le g_{2m-1}$ be the degrees of a minimal set of generators for $I_{G}$ and $(\epsilon_1,\epsilon_2,\epsilon_3):=\mci\beta_{\Gamma}.$ 

We start with the more delicate case, i.e. for $i=2$ and $j=3.$ By \cite{RZ3} Theorem 3.9, we have $\mci\beta_G\le\mci\beta_{\Gamma}\le (d_1,d_2,d_3),$ so if $\epsilon_2<d_2$ the conclusion follows for $d_2.$ Therefore we can suppose that $\epsilon_2=d_2.$
If $A$ is an Artinian Gorenstein graded algebra of codimension $3,$ whose Hilbert function is $H,$ we write $\Mng_H(n)$ as the largest number of minimal generators of degree $n$ compatible with an Artinian Gorenstein graded algebra of codimension $3$ having Hilbert function equal to $H$ (see Remark \ref{maxg}).
\par
Now, if $H$ is the Hilbert function of $A_{\Gamma},$ according to Proposition \ref{maxg} and Theorem \ref{c3b}, either $\epsilon_2=\gamma_2$ or $I_{\Gamma}$ has $\Mng_H(\epsilon_2)$ minimal generators of degree $\epsilon_2.$ But $\epsilon_2=\gamma_2$ implies that $\gamma_2=\epsilon_2=d_2=\vartheta_{\Gamma}-d_3$ is a degree of a first syzygy of $I_{\Gamma}$ which is clearly impossible. So, $I_{\Gamma}$ has $\Mng_H(\epsilon_2)$ minimal generators of degree $\epsilon_2=d_2.$ 
On the other hand the Hilbert 
function of $A_G$ is still $H$ and $I_G$ has $\Mng_H(d_2)-1$ minimal generators in degree $d_2.$ 
This implies, again by Proposition \ref{maxg} and Theorem \ref{c3b}, that $e_2<\epsilon_2=d_2.$ 
\par
Now, if $d_3>\epsilon_3$ we are done. So we can assume that $d_3=\epsilon_3.$ Observe that first  $\epsilon_3 > \gamma_3:$ indeed, otherwise, $d_3= \epsilon_3 = \gamma_3$ and consequently $d_2= \epsilon_2 = \gamma_2$ (we remind that $d_2$ and $d_3$ are degrees of minimal generators for $I_{\Gamma}$) and we should have again a first syzygy of degree $\gamma_2$ for $I_{\Gamma}.$ So we have two possibilities: either $\epsilon_2=\gamma_2$ or $\epsilon_2>\gamma_2.$ If $\epsilon_2=\gamma_2,$ since $\epsilon_3 > \gamma_3$, by Theorem \ref{c3b} $B_{\Gamma}=\emptyset$ and $C_{\Gamma}\ne \emptyset.$ If $\overline{B}_{\Gamma}=\emptyset$ then $I_{\Gamma}$ has $\Mng_H(\epsilon_3)-1$ minimal generators in degree $\epsilon_3.$ Then $e_2=\gamma_2$ and $I_G$ has $\Mng_H(\epsilon_3)-2$ generators in degree $\epsilon_3;$ therefore $e_3<\epsilon_3,$ i.e. $d_3>e_3.$ If $\overline{B}_{\Gamma}=\{n+2\}$ (see Remark \ref{Bvuoto}) then $\gamma_{n+2}\ge \vartheta_{\Gamma}-\gamma_{n+2}>\gamma_{n+1}$ hence $e_3\le g_{n+1}=\gamma_{n+1}<\gamma_{n+2}=\epsilon_3$ (note that, in this case,  $g_i=\gamma_i,$ for $i\le n+1$).

Finally, if $\epsilon_2>\gamma_2$ then 
$$B_{\Gamma}=\{3\le h\le m+1\mid\vartheta_{\Gamma}\le\gamma_h+\gamma_{2m+4-h}\}\ne \emptyset$$ 
(as defined in Theorem \ref{c3b}) and consequently $I_{\Gamma}$ has $\Mng_H(d_3)$ minimal generators in degree $d_3.$  Let
 $$B_{G}=\{3\le h\le m\mid\vartheta_{\Gamma}\le g_h+g_{2m+2-h}\}$$
and
 $$C_G=\{4\le h\le m+1\mid\vartheta_{\Gamma}\le g_h+g_{2m+3-h}\}$$
Now, if $B_G\ne\emptyset,$ since $I_G$ has $\Mng_H(d_3)-1$ minimal generators in degree $d_3,$ we see that $e_3<\epsilon_3=d_3.$ If $B_G=\emptyset$ and we set $b:=\max B_{\Gamma},$ there exist only two integers, $\gamma_b$ and $\gamma_{2m+4-b},$ in which $I_{\Gamma}$ has the maximum number of minimal generators allowed by $H.$ Therefore $\vartheta_{\Gamma}-\gamma_b=\gamma_{2m+4-b}=\epsilon_3=d_3$ and this implies also that $d_2=\gamma_b=\epsilon_2.$
Moreover $g_h=\gamma_h$ for $1\le h\le b-1,$ $g_h=\gamma_{h+1}$ for $b\le h\le 2m+2-b$ and $g_h=\gamma_{h+2}$ for $2m+3-b\le h\le 2m-1.$ Therefore      $$g_b=\gamma_{b+1}\ge\gamma_b=\vartheta_{\Gamma}-\gamma_{2m+4-b}\ge\vartheta_{\Gamma}-\gamma_{2m+5-b}=
\vartheta_{\Gamma}-g_{2m+3-b}$$
i.e. $b\in C_G.$ Let $c:=\max C_G.$ Of course $b\le c\le m+1\le 2m+2-b,$ so $e_3=g_c=\gamma_{c+1}.$ Therefore we have $g_c=\gamma_{c+1}=e_3 \le \epsilon_3=d_3=\gamma_{2m+4-b}.$  Suppose that $e_3= \epsilon_3=d_3,$ i.e. $\gamma_{c+1}=\gamma_{2m+4-b}.$ Since $c+1\le m+2<2m+4-b,$ we have  
$\gamma_h=\gamma_{2m+4-b}$ for $c+1\le h<2m+4-b.$ Consequently,
 $\gamma_{h}+\gamma_{2m+4-h}\ge\gamma_{2m+4-b}+\gamma_b=\vartheta_{\Gamma},$ i.e. $h$ and $2m+4-h$ satisfy the inequality to stay in $B_{\Gamma};$ but $h>b$ and $2m+4-h>b$ therefore $h\not\in B_{\Gamma}$ and $2m+4-h\not\in B_{\Gamma}$ so we get that $h\ge m+2$ and $2m+4-h\ge m+2$ i.e. $h=m+2.$ So we are reduced to the case $c+1=m+2$ and $2m+4-b=m+3$ i.e. $c=b=m+1$ and $\gamma_{m+2}=\gamma_{m+3}.$ So we have that
    $$\vartheta_{\Gamma}=\gamma_b+\gamma_{2m+4-b}=\gamma_{m+1}+\gamma_{m+3}=
    \gamma_{m+1}+\gamma_{m+2}<\vartheta_{\Gamma}$$
where the last inequality holds for the Gaeta-Diesel conditions. So we get a contradiction and consequently $e_3<\epsilon_3=d_3.$\\
The remaining cases are simpler. So, if $i=1$ and $j=2$ since $\vartheta_{\Gamma}-d_2=d_1$ is a first syzygy for $I_{\Gamma}$ we have $d_1>\epsilon_1\ge e_1.$ Moreover, $d_2=\vartheta_{\Gamma}-d_1\ge \vartheta_{\Gamma}/2\ge \gamma_{m+1}\ge \epsilon_2.$ thus, if $d_2=\epsilon_2$ we should have $d_2=\vartheta_{\Gamma}/2=d_1=\gamma_{m+1}$ and, consequently, $\vartheta_{\Gamma}/2=\vartheta_{\Gamma}-\gamma_{m+1}> \gamma_{m+2}\ge \vartheta_{\Gamma}/2,$ a contradiction, therefore $d_2>\epsilon_2\ge e_2.$\\
Finally, if $i=1$ and $j=3,$ as before we easily get that $d_1>\epsilon_1\ge e_1.$ Of course, we can suppose $d_3=\epsilon_3.$ We then show that $B_{\Gamma}\ne \emptyset:$ namely, otherwise,  $d_3=\epsilon_3\le \gamma_{m+2},$ let us say $d_3=\gamma_h$ for some $h\le m+2.$ Thus we should have
$$d_1=\vartheta_{\Gamma}-d_3=\vartheta_{\Gamma}-\gamma_h>\gamma_{2m+3-h}\ge \gamma_{m+1}\ge d_2$$
a contradiction. Therefore, we have $B_{\Gamma}\ne \emptyset$ and as in the case $i=2$ and $j=3$ we are reduced to the case in which $B_{\Gamma}$ has only one element $b.$ In this case we have $\epsilon_2=\gamma_b$ and $\epsilon_3=d_3=\gamma_{2m+4-b};$ then $d_1=\vartheta_{\Gamma}-d_3=\vartheta_{\Gamma}-\gamma_{2m+4-b}>\gamma_{b-1}$ and $\gamma_{b}\ge \vartheta_{\Gamma}-d_3=d_1$ since $b \in B_{\Gamma},$ i.e. $d_1=\gamma_b$ and the conclusion $e_3<\epsilon_3=d_3$ follows as in the final part of  the previous case $i=2$ and $j=3.$

\end{proof}

\begin{thm}\label{betti}
Let $(D,E,F)$ be a sequence of multisets of positive integers. 
$(D,E,F)$ is a Betti sequence admissible for an Artinian almost complete intersection algebra of codimension $3$ if and only if it satisfies the following conditions:
\begin{itemize}
	\item[1)]it is an \aci-type sequence;
	\item[2)]$\beta_G:=(G_0,G_1,G_2)$ is the Betti sequence of a $3$-codimensional Gorenstein Artinian graded algebra, where 
	$$G_0:=(\vartheta_Z-F)\sqcup\overline{D}\sqcup T$$
  $$G_1:=(-d_0+F)\sqcup(\vartheta_G-\overline{D})\sqcup T$$
  $$G_2:=\{\vartheta_G\}$$
with $\vartheta_G:=\vartheta_Z-d_0$ and $T:=\{\vartheta_G/2\}$ when $\vartheta_G/2\in \Supp S$ and $|F|+|\overline{D}|$ is even and  $T:=\emptyset$ otherwise.
\item[3)]if we set $(e_1,e_2,e_3):=\mci\beta_G,$ $e_1\le e_2\le e_3,$ $D^*=\{\{d_1,d_2,d_3\}\},$ $d_1\le d_2\le d_3$ then $d_i\ge e_i$ for $1\le i\le 3$ and for every $s\in \Supp (S\setminus T),$ $d_i>e_i$ for $i:=\min\{j\mid d_j=s\}+\mu_{S\setminus T}(s)-1.$
\end{itemize}
\end{thm}
\begin{proof}
Let us suppose that $(D,E,F)$ is the Betti sequence of an almost complete intersection of codimension $3.$ By Lemma \ref{l1} $(D,E,F)$ is an \aci-type sequence. Using the same terminology of the Definition \ref{aci} we denote $D=\{\{d_0,d_1,d_2,d_3\}\}$ with $d_0\le d_1\le d_2\le d_3.$ Let $A_Q=R/I_Q$ be an almost complete intersection Artinian graded algebra with Betti sequence $(D,E,F)$ and $A_Z=R/I_Z$ be a complete intersection Artinian graded algebra of type $(d_1,d_2,d_3)$ with $I_Z$ generated by $f_1,f_2,f_3$ minimal generators of $I_Q.$ Note that such $A_Z$ must exist since $d_1,d_2,d_3\ge d_0.$ Let $I_{\Gamma}:=I_Z:I_Q.$ Of course, $A_{\Gamma}=R/I_{\Gamma}$ is a Gorenstein Artinian graded algebra. Using Remark \ref{minres} a minimal graded free resolution of $A_{\Gamma}$ will be of the type
$$
\begin{array}{c}
\displaystyle{\bigoplus_{h\in \vartheta_Z-F }R(-h)\oplus \bigoplus_{h\in \overline{D}}R(-h)	\oplus \bigoplus_{h\in \overline{S}}R(-h)}\\
\uparrow\\
\displaystyle{\bigoplus_{h\in -d_0+F }R(-h)\oplus \bigoplus_{h\in\vartheta_G- \overline{D}}R(-h)	\oplus \bigoplus_{h\in \overline{S}}R(-h)}\\
\uparrow\\
R(-\vartheta_{\Gamma})\\
\uparrow\\
0
\end{array}
$$
Using Remark \ref{minres} one sees that either $\overline{S}=T$ or $\overline{S}\setminus T=\{\{\alpha, \vartheta_G-\alpha\}\}$ for some $\alpha;$ therefore if we substitute in the previos resolution $\overline{S}$ with $T$ we get a graded  minimal free resolution of a Gorenstein Artinian graded algebra $A_G$ with $\vartheta_G=\vartheta_{\Gamma}$ (see \cite{RZ1} Remark 3.8). Hence condition $2)$ is verified. Now observe that $\beta_{G} \le \beta_{\Gamma}$ hence $\mci \beta_{G}\le \mci \beta_{\Gamma}$ (see \cite{RZ3} Theorem 3.9). Since $I_Z \subseteq I_{\Gamma}$ $(d_1,d_2,d_3)\ge \mci \beta_{\Gamma} \ge \mci \beta_G=(e_1,e_2,e_3).$ 
Let $s \in \Supp(S\setminus T)$ and $i=\min\{j\mid d_j=s\}+\mu_{S\setminus T}(s)-1,$  say $m:=\mu_{(S\setminus T)}(s).$
Then, among $f_1,f_2,f_3,$ there are $m$ forms, of degree $s,$ such that each of them is either a generator not minimal for $I_{\Gamma}$ or it is a minimal generator for $I_{\Gamma},$ but there exists a minimal generator for $I_{\Gamma}$ of degree $\vartheta_{\Gamma}-s.$
So, by Lemma~\ref{magg} and Lemma~\ref{zzz}, $|\{j\mid d_j=s,\,d_j>e_j\}|\ge m.$ Let $h:=\max\{j\mid d_j=s,\,d_j>e_j\}.$ Then 
 $$h\ge\min\{j\mid d_j=s,\,d_j>e_j\}+m-1\ge i:=\min\{j\mid d_j=s\}+m-1,$$
therefore $s=d_i=d_h>e_h\ge e_i$ and we are done.

Vice versa, since $(D,E,F)$ is \aci-type we have 
\begin{itemize}
	\item[1)]$D=\{d_0\} \sqcup \overline{D} \sqcup S,$ $|\overline{D} \sqcup S|=3;$
	\item[2)]$E=(d-F)\sqcup (d_0+\overline{D})\sqcup (\vartheta_Z-S).$ 	
\end{itemize}
Let $A_G=R/I_G$ be a $3$-codimensional Artinian Gorenstein algebra with Betti sequence $\beta_G=(G_0,G_1,G_2);$ by condition 3), using part 2) of Lemma \ref{magg} for every  $s \in \Supp (S\setminus T)$ with multiplicity $\mu_{S\setminus T}(s)=m$ we can find a regular sequence of lenght $3$ in $I_G$ with $m$ elements of degree $s$ which are not minimal generators for $I_G$ (and the other elements minimal generators for $I_G$). Say $I_Z$ the complete intersection generated by this regular sequence and $I_Q:=I_Z:I_G;$ by mapping cone procedure we see that the minimal resolution of $I_Q$ will be of the following type:

$$
\begin{array}{c}
\displaystyle{R(-d_0)\oplus \bigoplus_{h\in \overline{D}}R(-h)	\oplus \bigoplus_{h\in S}R(-h)}\\
\uparrow\\
\displaystyle{\bigoplus_{h\in d-F }R(-h)\oplus \bigoplus_{h\in d_0+ \overline{D}}R(-h)	\oplus \bigoplus_{h\in \vartheta_Z-S}R(-h)}\\
\uparrow\\
\displaystyle{\bigoplus_{h\in F }R(-h)}\\
\uparrow\\
0
\end{array}
$$
which means that $(D,E,F)$ is a Betti sequence of an almost complete intersection.
\end{proof}

\begin{exm}
In this example we produce a sequence in which all the conditions of Theorem \ref{betti} hold but the last one, so it is not admissible as a Betti sequence for an  almost complete intersection. Let $D=\{\{3,6,6,6\}\},$ $E=\{\{8,8,8,10,10,10,12,12,12,12\}\},$ $F=\{\{9,11,11,11,13,13,13\}\}$ and $\beta=(D,E,F).$ Then $\vartheta_Z=18,$ $\vartheta_G=15,$ $d=21,$ $D^*=\{\{6,6,6\}\},$ $S=\{\{6,6,6\}\},$ $G_0=\{\{5,5,5,7,7,7,9\}\},$ $G_1=\{\{6,8,8,8,10,10,10\}\},$ $G_2=\{15\}.$ Note that $\beta_G:=(G_0,G_1,G_2)$ is a Gorenstein Betti sequence with $\mci\beta_G=(5,5,7)$ hence $(6,6,6)\not\ge \mci\beta_G.$
\end{exm}

\begin{exm}
In this example we produce a sequence in which all the conditions of Theorem \ref{betti} hold, except that one regarding the $i=\min\{j\mid d_j=s\}+\mu_{S\setminus T}(s)-1$ for some $s\in \Supp (S\setminus T),$  so it is not admissible as a Betti sequence for an almost complete intersection. Let $D=\{\{2,5,5,7\}\},$ $E=\{\{7,7,7,9,9,9,10,11\}\},$ $F=\{\{8,10,10,10,12\}\}$ and $\beta=(D,E,F).$ Then $\vartheta_Z=17,$ $\vartheta_G=15,$ $d=19,$ $D^*=\{\{5,5,7\}\},$ $S=\{7\},$ $G_0=\{\{5,5,5,7,7,7,9\}\},$ $G_1=\{\{6,8,8,8,10,10,10\}\},$ $G_2=\{15\}.$ Note that $\beta_G:=(G_0,G_1,G_2)$ is a Gorenstein Betti sequence with $\mci\beta_G=(5,5,7),$ $(d_1,d_2,d_3)=(5,5,7)\ge\mci\beta_G,$ but $3=\min\{j\mid d_j=7\}+\mu_{S\setminus T}(7)-1$  and $d_3=7$ is not greater than $e_3=7.$
\end{exm}

\begin{exm}
The following sequences are admissible as Betti sequences for an  almost complete intersection:
 $$(D,E,F)=\big(\{4,5^{(2)},9\},\{9^{(3)},11^{(3)},13\},\{12^{(3)},14\}\big)$$
and
 $$(D,E,F)=\big(\{4,5^{(2)},9\},\{9^{(3)},10,11^{(3)},13\},\{10,12^{(3)},14\}\big).$$
Note that we can obtain the first one by the second one by deleting the \lq\lq ghost\rq\rq\ degree $10.$
\end{exm}

\vspace{1cm}
{\f
{\sc (A. Ragusa) Dip. di Matematica e Informatica, Universit\`a di Catania,\\
                  Viale A. Doria 6, 95125 Catania, Italy}\par
{\it E-mail address: }{\tt ragusa@dmi.unict.it} \par
{\it Fax number: }{\f +39095330094} \par
\vspace{.3cm}
{\sc (G. Zappal\`a) Dip. di Matematica e Informatica, Universit\`a di Catania,\\
                  Viale A. Doria 6, 95125 Catania, Italy}\par
{\it E-mail address: }{\tt zappalag@dmi.unict.it} \par
{\it Fax number: }{\f +39095330094}
}

\end{document}